\documentclass[11pt,a4paper]{article}
\usepackage[utf8]{inputenc}
\usepackage{amsmath}
\usepackage{amsfonts}
\usepackage{amssymb}
\usepackage{enumitem}
\usepackage{amsthm}
\usepackage{graphicx}
\usepackage{setspace}
\usepackage{tikz}
\usepackage{subcaption}

\usepackage[top=1.42in, bottom=1.42in, left=1in, right=1in]{geometry}

\newtheorem{theorem}{Theorem}[section]

\newtheorem{lemma}[theorem]{Lemma}

\newtheorem{conjecture}[theorem]{Conjecture}

\newtheorem{corollary}[theorem]{Corollary}
\newtheorem{definition}[theorem]{Definition}

\newcommand{\vmin}{v^-_\lozenge}
\newcommand{\vplus}{v^+_\lozenge}

\usepackage{xspace}
\newcommand{\ttg}{two-terminal graph\xspace}
\newcommand{\ttgs}{two-terminal graphs\xspace}
\newcommand{\dip}{dipole\xspace}
\newcommand{\dips}{dipoles\xspace}

\usepackage[normalem]{ulem}

\def\({\left(}
\def\){\right)} 
\def\[{\left[}
\def\]{\right]}
\newcommand{\R}{\mathbb{R}}

\setlist{itemsep=0pt,topsep=5pt}

\usepackage{comment}

\title{Density of Zeros of the Tutte Polynomial }
\author{Seongmin Ok\thanks{Korea Institute for Advanced Study, email: \texttt{seong@kias.re.kr}}\; and Thomas J. Perrett\thanks{Technical University of Denmark, email: \texttt{tper@dtu.dk}. Supported by ERC Advanced Grant GRACOL, project number 320812.}}
\date{\today}

\begin{document}
\maketitle

\begin{abstract}
The Tutte polynomial of a graph is a two-variable polynomial whose zeros and evaluations encode many interesting properties of the graph. In this article we investigate the zeros of the Tutte polynomials of graphs, and show that they form a dense subset of certain regions of the plane. This is the first density result for the zeros of the Tutte polynomial in a region of positive volume. Our result almost confirms a conjecture of Jackson and Sokal except for one region which is related to an open problem on flow polynomials.
\end{abstract}

\section{Introduction}	
Let $G$ be a graph with vertex set $V$ and edge set $E$. We allow $G$ to have loops and multiple edges. In this paper we consider the random cluster formulation of the Tutte polynomial of $G$, which is defined to be the polynomial

\begin{equation}\label{eq:BVT}
Z_G(q,v) = \sum_{A \subseteq E} q^{k(A)} v^{|A|},
\end{equation}
where $q$ and $v$ are commuting indeterminates, and $k(A)$ denotes the number of components in the graph $(V,A)$. We can retrieve the classical formulation of the Tutte polynomial $T_G(x,y)$ from $Z_G(q,v)$ by a simple change of variables as follows:

$$T_G(x,y) = (x-1)^{-k(E)}(y-1)^{-|V|} Z_G\big( (x-1)(y-1),y-1 \big).$$
Since we deal mainly with the formulation given by~\eqref{eq:BVT}, we will say that $Z_G(q,v)$ is the \emph{Tutte polynomial} of $G$, and $T_G(x,y)$ is the \emph{classical Tutte polynomial} of $G$.

The chromatic polynomial $P_G(q)$ is a well-studied specialisation of the Tutte polynomial, and can be obtained from~\eqref{eq:BVT} by setting $v = -1$. In particular, the zeros of the chromatic polynomials of graphs are of much interest due to their connection with graph colouring and statistical physics, see~\cite{Sok.MultiTutte}.  Notably, Tutte~\cite{Tut.01} showed that no zeros of the chromatic polynomials of graphs lie in the set $(-\infty, 0) \cup (0,1)$, and Jackson~\cite{Jac.32/27} showed that the same conclusion holds for the interval $(1, 32/27]$. On the other hand, Thomassen~\cite{Tho.RootsDense} proved the complementary result that such zeros are dense in the interval $(32/27, \infty)$. We also mention the result of Sokal~\cite{Sok.ComplexDense} which says that the \emph{complex} zeros of the chromatic polynomial form a dense subset of the complex plane.

Expanding on Jackson's study, Jackson and Sokal~\cite{Jac.Sok.ZeroFreeMultiTutte} identified a \emph{zero-free region} $R_1$ of the $(q,v)$ plane where the Tutte polynomial never has a zero. Using a multivariate version of the Tutte polynomial, their work led to a much better understanding of such zero-free regions, and further elucidated the origin of the number $32/27$. They conjectured that $R_1$ is the first in an inclusion-wise increasing sequence of regions $R_1, R_2, \dots$, such that for $i \geq 1$, the only $2$-connected graphs whose Tutte polynomials have a zero inside $R_i$ have fewer than $i$ edges. Jackson and Sokal also conjectured that this sequence converges to a limiting region $R^*$, outside of which the zeros of the Tutte polynomials of graphs are dense. The region $R^*$ is depicted by the unshaded region in Figure~\ref{fig:introBoth}.

We now state the conjecture of Jackson and Sokal precisely. Following~\cite{Jac.Sok.ZeroFreeMultiTutte}, let $\vplus(q)$ be the function describing the middle branch of the curve $v^3 - 2qv - q^2 = 0$ for $0 < q < 32/27$, see Figure~\ref{fig:introBoth} or~\cite[Figure 2]{Jac.Sok.ZeroFreeMultiTutte}. Also, let $\vmin(q)$ be defined by $\vmin(q) = q/ \vplus(q)$ for $0 < q < 32/27$.

\begin{conjecture}\emph{\cite{Jac.Sok.ZeroFreeMultiTutte}}\label{conj:main}
The zeros of the Tutte polynomials of graphs are dense in the following regions:
\begin{enumerate}[label=(\alph*)]
\item $q < 0$ and $v < -2$,
\item $q < 0$ and $0 < v < -q/2$,
\item $0 < q \leq 32/27$ and $v < \vmin$,
\item $0 < q \leq 32/27$ and $\vplus < v < 0$, and
\item $q > 32/27$ and $v < 0$.
\end{enumerate}
\end{conjecture}

\begin{figure}
\centering
\includegraphics[scale=1]{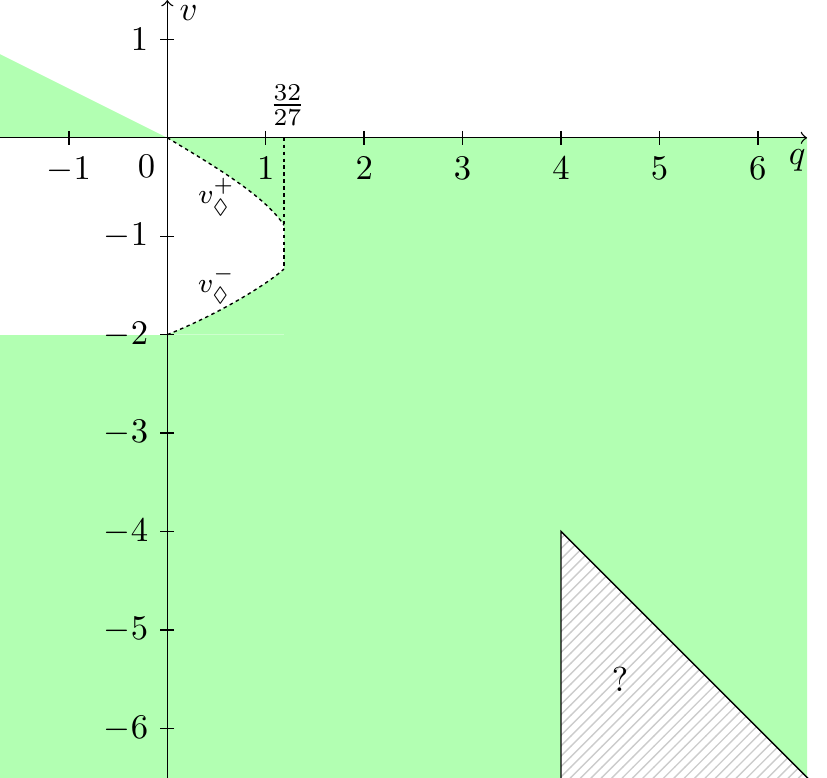}
\caption{Regions in Conjecture~\ref{conj:main}}\label{fig:introBoth}
\end{figure}

The union of the regions described in Conjecture~\ref{conj:main} is illustrated by the shaded and hatched area in Figure~\ref{fig:introBoth}.

In this paper, we prove Conjecture \ref{conj:main} in many cases. Our main tool is a technique of Thomassen~\cite{Tho.RootsDense}, which was originally developed for the chromatic polynomials of graphs, and which we extend to the Tutte polynomial.

\begin{theorem}\label{thm:main}
The zeros of the Tutte polynomials of graphs are dense in the following regions:
\begin{enumerate}[label=(\alph*)]
\item $q < 0$ and $v < -2$,
\item $q < 0$ and $0 < v < -q/2$,
\item $0 < q \leq 32/27$ and $v < \vmin$,
\item $0 < q \leq 32/27$ and $\vplus < v < 0$, 
\item $32/27 < q < 4$ and $v < 0$, and
\item $q > 4$ and $-q < v < 0$.
\end{enumerate}
\end{theorem}

Thus, the only region of Jackson and Sokal's conjecture which is not covered by Theorem~\ref{thm:main} is the region defined by $q > 4$ and $v < - q$. This region is indicated by a question mark in Figure~\ref{fig:introBoth}. We later discuss the obstructions that arise in this region which are related to an open problem on the flow polynomials of graphs.

\section{The Multivariate Tutte Polynomial}

In this section we introduce the multivariate Tutte polynomial and briefly describe the advantages in using this more general version. We refer the reader to~\cite{Sok.MultiTutte} for a more detailed introduction. Let $G$ be a graph with vertex set $V$ and edge set $E$. The \emph{multivariate Tutte polynomial} of $G$ is the polynomial
\begin{equation}\label{eq:MVT}
Z_G(q,\textbf{v}) = \sum_{A \subseteq E} q^{k(A)} \prod_{e \in A} v_e,
\end{equation}
where $q$ and $\textbf{v} = \{v_e\}_{e \in E}$ are commuting indeterminates, and $k(A)$ denotes the number of components in the graph $(V,A)$. We say that $(G, \textbf{v})$ is a \emph{weighted graph} and $v_e$ is the \emph{weight} of the edge $e$. It can be seen from~\eqref{eq:BVT} that the random cluster formulation of the Tutte polynomial is obtained by setting all edge weights equal to a single indeterminate $v$. Despite being interested in the zeros of the Tutte polynomial, we will find it useful to consider the multivariate version. This point of view has proven to be particularly useful in studying the computational complexity of the evaluations \cite{Gol.Jer.Inapprox,Gol.Jer.InapproxPlanar,Gol.Jer.ComplexityTutte,Jae.Ver.Wel.Compl.Tutte}, and the zeros of the Tutte polynomial \cite{Jac.Sok.ZeroFreeMultiTutte,Sok.ComplexDense}.

The major advantage of using the multivariate Tutte polynomial is that, in certain circumstances, one can replace a subgraph by a single edge with an appropriate weight. Indeed, suppose $(G, \textbf{v})$ is a weighted graph and $F$ and $H$ are connected subgraphs of $G$ such that $V(F) \cup V(H) = V(G)$ and $V(F) \cap V(H) = \{x, y\}$. Let $F_{xy}$ denote the graph formed by identifying the vertices $x$ and $y$ in $F$ (the edges between $x$ and $y$ become loops) and let $F+xy$ denote the graph formed by adding an edge $xy$ to $F$. We also let $Z_{F+xy}(q ,\textbf{v}|_F,w)$ denote the multivariate Tutte polynomial of $F+xy$ where the new edge $xy$ has weight $w$. Finally, let $v_F$ denote the \emph{effective weight} of $F$ in $(G, \textbf{v})$, which is defined by

\begin{equation}\label{eq:effW}
1 + v_F =\frac{(q-1)Z_{F_{xy}}(q, \textbf{v}|_F)}{Z_{F+xy}(q, \textbf{v}|_F, -1)} = \frac{(q-1)Z_{F_{xy}}(q, \textbf{v}|_F)}{Z_{F}(q, \textbf{v}) - Z_{F_{xy}}(q, \textbf{v}|_F)}.
\end{equation}

If $F$ is a graph with $x,y \in V(F)$, then we write $v_F(q,v)$ to indicate the effective weight of the graph $F$ where every edge has weight $v$. The following lemma shows that replacing the graph $F$ with a single edge of weight $v_F$ only changes the multivariate Tutte polynomial by a prefactor depending on $F$.

\begin{lemma}\emph{\cite{Don.Jac.Nearly3Con}}\label{lem:effW}
$Z_G(q, \textbf{v}) = \frac{1}{q(q-1)}Z_{F+xy}(q, \textbf{v}|_F,-1) Z_{H + xy}(q, \textbf{v}|_H, v_F).$
\end{lemma}

We briefly note the effective weights of two common graphs which we will frequently use. See~\cite{Sok.MultiTutte} for a more detailed derivation. We define a \emph{\dip} to be the loopless multigraph consisting of two vertices $x$ and $y$ connected with a number of parallel edges. Thus a single edge $xy$ is considered to be a dipole. If $F$ is a dipole with $s$ edges of weights $v_1, \dots, v_s \in \mathbb{R}$, then the effective weight $v_F$ of $F$ satisfies 

\begin{equation}\label{eq:effWdipole}
1 + v_F = \prod_{i=1}^s (1 + v_i).
\end{equation}

We also use $P_s$ to denote the path with $s$ edges whose end vertices are labelled $x$ and $y$. If the edges of $P_s$ have weights $v_1, \dots, v_s \in \mathbb{R}$, then the effective weight $v_{P_s}$ of $P_s$ satisfies

\begin{equation}\label{eq:effWpath}
1 + \frac{q}{v_{P_s}} = \prod_{i=1}^s \(1 + \frac{q}{v_i}\).
\end{equation}

We say that a connected loopless graph $F$ with two vertices labelled $x$ and $y$ respectively is a \emph{\ttg} if $x$ and $y$ are not adjacent in $F$. The vertices $x$ and $y$ are called \emph{terminals}. The following lemma shows that these graphs satisfy a technical condition which will be required later.

\begin{lemma}\label{lem:effWpositive}
Let $v \in \mathbb{R}$ be fixed. If $F$ is a \ttg, then $v_F(q,v)$ is not a constant function of $q$ and there exists $q_0 > 0$ such that $1 + v_F (q,v) > 0$ for all $q > q_0$.
\end{lemma}

\begin{proof}
By~\eqref{eq:effW}, we have
$$1 + v_F = \frac{(q-1)Z_{F_{xy}}(q,v)}{Z_{F+xy}(q,v,-1)} = \frac{qZ_{F_{xy}}(q,v) - Z_{F_{xy}}(q,v)}{Z_F(q,v) - Z_{F_{xy}}(q,v)}.$$
Since $xy \not\in E(F)$, the graph $F_{xy}$ is loopless. Hence, the terms with the highest powers of $q$ in $Z_{F_{xy}}(q,v)$ and $Z_{F+xy}(q,v,-1)$ are $q^{|V(F)|-1}$ and $q^{|V(F)|}$ respectively. Thus there exists such $q_0 > 0$.

By the same reason, if $1+v_F$ is a constant function of $q$ then it is 1, and $qZ_{F_{xy}}(q,v) = Z_F(q,v)$. However, $Z_F(q,v)$ has a nonzero coefficient of $q$ whereas the minimum degree of $q$ in $qZ_{F_{xy}}$ is 2. This is a contradiction, so $1+v_F$ is not a constant function of $q$.
\end{proof}

\section{Strategy}\label{sec:Strat}

In~\cite{Tho.RootsDense}, Thomassen showed that the zeros of the chromatic polynomial include a dense subset of the interval $(32/27, \infty)$. In this section we generalise his technique to the Tutte polynomial. At the heart of Thomassen's method lies the following lemma, which is proved implicitly in Proposition 2.3 of~\cite{Tho.RootsDense}.

\begin{lemma}\emph{\cite{Tho.RootsDense}}\label{lem:functions}
Let $I \subseteq \mathbb{R}$ be an interval of positive length, and let $a, b$ and $c$ be rational functions of $q$ such that $0 < b(q) < 1 < a(q)$ and $c(q)>0$ for all $q \in I$. If there is no $\alpha \in \mathbb{Q}$ such that $\log(a(q)) / \log(b(q))  = \alpha$ for all $q \in I$, then there exist natural numbers $s$ and $t$ such that $a(q_0)^s b(q_0)^t = c(q_0)$ for some $q_0 \in I$. Moreover, $s$ and $t$ can be chosen to have prescribed parity.
\end{lemma}

Let $F$ be a \ttg or \dip. For real numbers $q$ and $v$, we say that $F$ has one of four types at $(q,v)$ defined by the following conditions.
\begin{list}{}{}
\item[\textbf{Type $A^+$:}] $1 + v_F (q,v) > 1$, 
\item[\textbf{Type $A^-$:}] $1 + v_F (q,v) < -1$,
\item[\textbf{Type $B^+$:}] $0 < 1 + v_F  (q,v) < 1$,
\item[\textbf{Type $B^-$:}] $-1 < 1 + v_F (q,v) < 0$.
\end{list}

Let $A$ and $B$ be \ttgs or \dips. Note that if $A$ is a graph of type $A^+$ or $A^-$ at $(q,v)$, then the rational function $1+ v_A(q,v)$ or $-1- v_A(q,v)$ can play the role of $a(q)$ in Lemma~\ref{lem:functions}. The corresponding property holds for a graph of type $B^+$ or $B^-$. 

\begin{definition}\label{def:complementaryPair}
Let $A$ and $B$ be \ttgs or \dips. We say that the pair $(A,B)$ is \emph{complementary} at $(q,v)$ if at most one of $A$ and $B$ is a \dip, and 
\begin{enumerate}[label=-]
\item $A$ has type $A^+$ and $B$ has type $B^-$ at $(q,v)$, or
\item $A$ has type $A^-$ and $B$ has type $B^+$ at $(q,v)$.
\end{enumerate}
\end{definition}

The definition of complementary pairs is motivated by the following key lemma.

\begin{lemma}\label{lem:rootFind}
Let $q_0, v \in \mathbb{R}$ be fixed such that  $q_0 \neq 1$, and let $A$ and $B$ be \ttgs or \dips. If $(A,B)$ is complementary at $(q_0,v)$, then for all $\varepsilon > 0$, there is a graph $G$ such that $Z_G(q_1,v) = 0$ for some $q_1 \in (q_0-\varepsilon,q_0+\varepsilon)$. Furthermore, if both $A+xy$ and $B+xy$ are planar, then we can choose $G$ to be planar.
\end{lemma}

\begin{proof}
Suppose that $A$ has type $A^-$ and $B$ has type $B^+$ at $(q_0,v)$. The other case is analogous. By continuity of $v_A(q,v)$ and $v_B(q,v)$, there exists an interval $I \subseteq \mathbb{R}$ of small positive length such that $q_0 \in I \subseteq (q_0-\varepsilon, q_0 + \varepsilon)$ and the graphs $A$ and $B$ have types $A^-$ and $B^+$ respectively for all $q \in I$. For positive integers $s$ and $t$ (which will be determined later), let $G$ denote the graph obtained from $s$ copies of $A$ and $t$ copies of $B$ by identifying all vertices labelled $x$ into a single vertex and all vertices labelled $y$ into another. In other words, we place all copies of $A$ and $B$ in parallel. Notice that by construction the graph $G$ is planar if $A+xy$ and $B+xy$ are both planar. Using Lemma~\ref{lem:effW} and~\eqref{eq:effWdipole}, one can see that the Tutte polynomial of $G$ is
$$Z_{G}(q,v) = q \left(\frac{Z_{A+xy}(q,v,-1)}{q(q-1)}\right)^s \left(\frac{Z_{B+xy}(q,v,-1)}{q(q-1)}\right)^t f(q,v),$$
where $f(q,v) = q-1 + (1+v_A(q,v))^s(1+v_B(q,v))^t.$ Define $a(q) = - 1 - v_A(q,v)$ and $b(q) = 1 + v_B(q,v)$. Moreover, define $c(q) = q - 1$ or $c(q) = 1 - q$ such that $c(q)>0$ for all $q \in I$. In doing this it may be necessary to replace $I$ with an appropriate subinterval. If $a$, $b$ and $c$ satisfy the conditions in Lemma~\ref{lem:functions}, then there are positive integers $s$ and $t$, of any prescribed parity, such that $a(q_1)^s b(q_1)^t = c(q_1)$ for some $q_1 \in I$. Since $f(q,v)$ is $c(q) + (-a(q))^s b(q)^t$ or $- c(q) + (-a(q))^s b(q)^t$, we may choose the parity of $s$ and $t$ such that this factor becomes zero for some $q_1 \in I$. 

It remains to check that $a$, $b$ and $c$ satisfy the conditions of Lemma~\ref{lem:functions}. Indeed, by assumption, we have that $0 < b(q) < 1 < a(q)$ and $c(q) > 0$ for all $q \in I$. Suppose for a contradiction that there is $\alpha \in \mathbb{Q}$ such that $\log(a(q)) / \log(b(q)) = \alpha$ for all $q \in I$. Equivalently $a(q) = b(q)^\alpha$ for $q \in I$. Since $a$ and $b$ are rational functions and $\alpha \in \mathbb{Q}$, it follows that this equality is satisfied for all $q \in \mathbb{R}$ except for any singularities. Also, since $a(q) > 0$ for $q \in I$, we take the principal branch of any fractional power. Since $(A,B)$ is complementary, at most one of $A$ and $B$ is a \dip, and thus by Lemma~\ref{lem:effWpositive}, at most one of $a$ and $b$ is a constant function. If precisely one of $a$ and $b$ is a constant function, then we immediately deduce a contradiction. Thus, we may assume that both of $A$ and $B$ are \ttgs. By Lemma~\ref{lem:effWpositive} we see that $a(q) < 0$ and $b(q) > 0$ for large enough $q$. This contradicts the assertion that $a(q) = b(q)^ \alpha$ for all $q \in \mathbb{R}$.
\end{proof}

\begin{corollary}\label{cor:strategy}
If $R$ is an open subset of the $(q,v)$ plane such that for every $(q, v) \in R$ there is a complementary pair of graphs $(A,B)$, then the zeros of the Tutte polynomials of graphs are dense in $R$.
\end{corollary}

To obtain density in some regions we will use planar duality. Let $G$ be a plane graph and let $G^*$ be its planar dual. The following relation is easily derived from~\eqref{eq:MVT} and Eulers formula, see~\cite{Sok.MultiTutte}.
\begin{equation}\label{eq:Duality}
Z_{G^*}(q,v) = q^{1-n}v^m Z_{G}(q,\tfrac{q}{v}).
\end{equation}

Thus we have a second corollary of Lemma~\ref{lem:rootFind}. Let us call a complementary pair $(A,B)$ \emph{planar} if both $A+xy$ and $B+xy$ are planar.

\begin{corollary}\label{cor:duality}
Let $R$ be an open subset of the $(q,v)$ plane. If for every $(q,v) \in \R$ there is a planar complementary pair of graphs, then the zeros of the Tutte polynomials of graphs are dense in $R^*$ where
\begin{equation}\label{eq:dualRegion}
R^* = \{(q,v) : (q,q/v) \in R \}.
\end{equation}
\end{corollary}

\section{Complementary Pairs}

In this section, we find complementary pairs of graphs for points $(q,v)$ in several regions of the $(q,v)$ plane. Combining this with Corollaries~\ref{cor:strategy} and~\ref{cor:duality}, we deduce Theorem~\ref{thm:main}. In what follows it will be useful to partition the $(q,v)$ plane into a number of regions, see Figures~\ref{fig:outsideBox} and~\ref{fig:insideBox}. Note that taken together, the closure of the regions below is equal to the union of the regions in Theorem~\ref{thm:main}. Thus, if the zeros of the Tutte polynomials of graphs are dense in these regions, then Theorem~\ref{thm:main} follows.

\begin{itemize}[label={}]
\item Region I: $q < 0$ and $v < -2$.
\item Region II: $0 < q < 1$ and $v < -2$.
\item Region III: $1 < q < 2$ and $v < -2$.
\item Region IV: $2 < q < 4$, $q \neq 3$ and $v < -q$.
\item Region V: $q>2$ and $-q < v < -2$.
\item Region VI: $2 < q < 4$ and $-2 < v < -q/2$.
\item Region VII: $q > 2$ and $- 1 < v < 0$.
\item Region VIII: $0 < q < 32/27$ and $-2 < v < \vmin$.
\item Region IX: $32/27 < q < 2$ and $-2 < v < -1$.
\end{itemize}

We also define the following `dual' regions in the sense of~\eqref{eq:dualRegion}. It is easy to check that the dual of each region below is contained in the corresponding region above.

\begin{itemize}[label={}]
\item Region I$^*$: $q < 0$ and $0 < v < -q/2$.
\item Region II$^*$: $0 < q < 1$ and $-q/2 < v < 0$.
\item Region III$^*$: $1 < q < 2$ and $-q/2 < v < 0$.
\item Region V$^*$: $-2 < v < -1$ and $q > -2v$.
\item Region VIII$^*$: $0 < q < 32/27$ and $\vplus < v < -q/2$.
\item Region IX$^*$: $32/27 < q < 2$ and $-1 < v < -q/2$.
\end{itemize}

\begin{figure}
\centering
\includegraphics[scale=1]{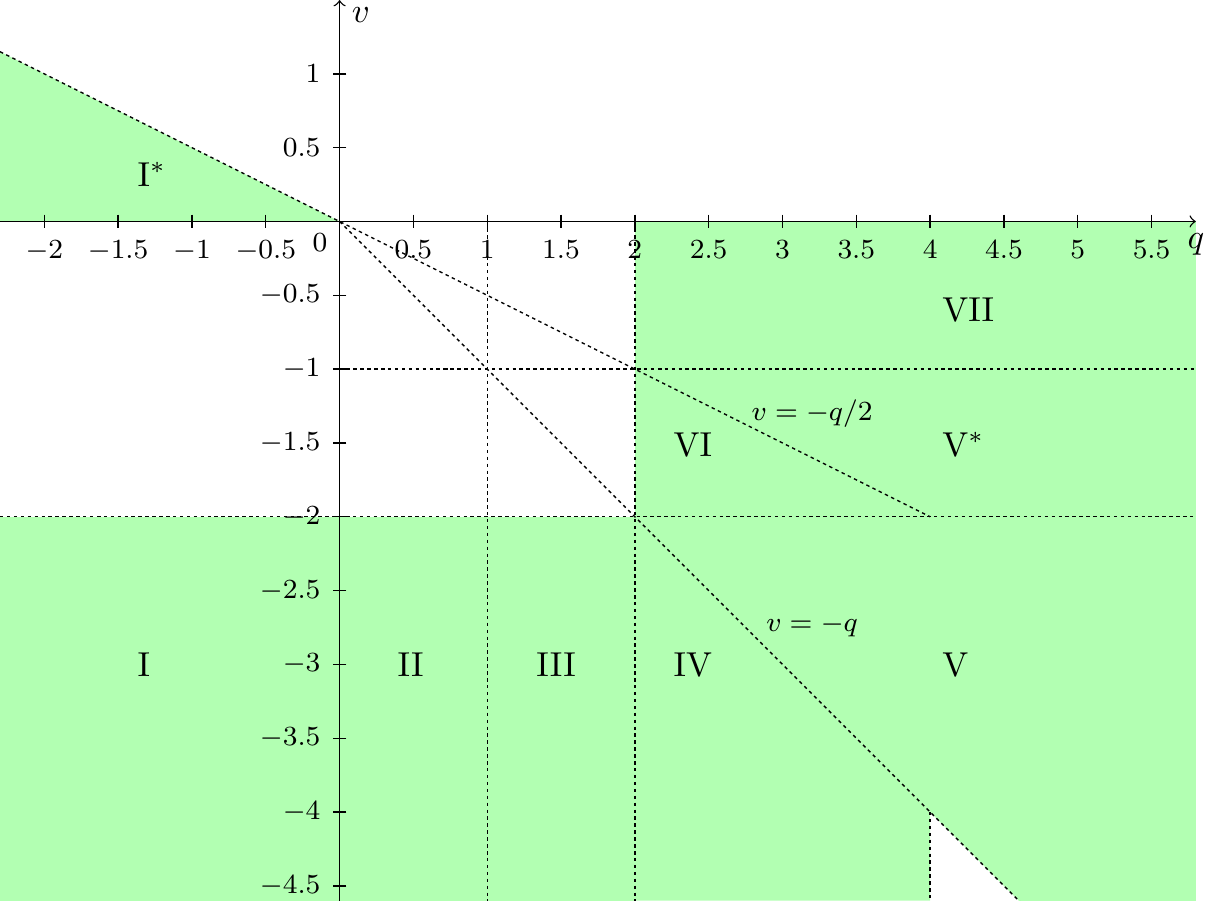}
\caption{Regions I - VII, I$^*$ and V$^*$.}\label{fig:outsideBox}
\end{figure}

\begin{figure}
\centering
\includegraphics[scale=1]{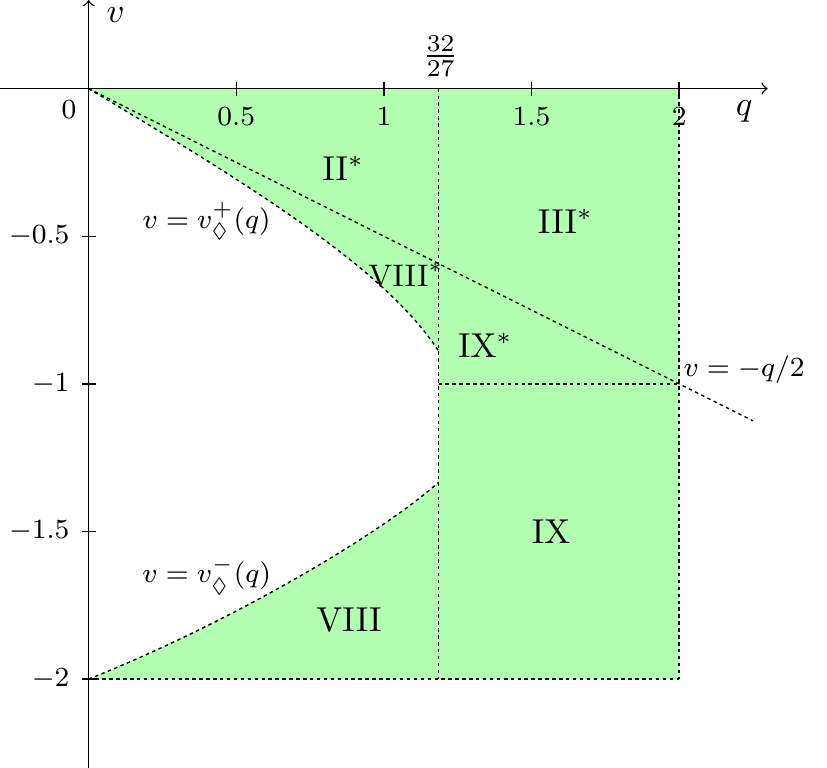}
\caption{Regions VIII, IX, II$^*$, III$^*$, VIII$^*$ and IX$^*$}\label{fig:insideBox}
\end{figure}

In the following lemma, we show that the path of an appropriate length $s$ gives graphs of varying types depending on the point $(q,v)$. The condition $s > 1$ guarantees that the resulting graphs are \ttgs. We note that parts~\ref{paths:(SM1)} and~\ref{paths:(SM2)} are equivalent to Lemmas 21 and 22 in~\cite{Gol.Jer.ComplexityTutte} respectively. 

\begin{lemma}\label{lem:paths}
Let $q$ and $v$ be real numbers, and let $P_s$ denote the path of length $s$ where every edge has weight $v$.
\begin{enumerate}[label=(\roman*)]
\item\label{paths:(5)} If $v < 0$ and $q < 0$, then there exits $s > 1$ such that $P_s$ has type $B^+$ at $(q,v)$.
\item\label{paths:(SM1)} If $v < -2$ and $0 < q < 1$, then there exits $s > 1$ such that $P_s$ has type $B^+$ at $(q,v)$.
\item\label{paths:(SM2)} If $v < -2$ and $1 < q < 2$, then there exits $s > 1$ such that $P_s$ has type $B^-$ at $(q,v)$.
\item\label{paths:(1)} If $v < 0$ and $q > -2v$, then there exist $s > 1$ such that $P_s$ has type $A^+$ at $(q,v)$.
\item\label{paths:(VI)} If $v < 0$ and $2 < q < - 2v$, then there exits $s > 1$ such that $P_s$ has type $A^-$ at $(q,v)$.
\end{enumerate}
\end{lemma}

\begin{proof}
By~\eqref{eq:effWpath}, the effective weight of $P_s$ is given by
\begin{equation}\label{eq:path}
v_{P_{s}} = \frac{q}{(1+\frac{q}{v})^s -1}.
\end{equation}

\begin{enumerate}[label=(\roman*)]
\item Since $q/v > 0$, the denominator of~\eqref{eq:path} tends to infinity with $s$. Because $q < 0$, we deduce that there exists $s > 1$ such that $-1 < v_{P_s} < 0$. Thus, $P_s$ has type $B^+$ as claimed.
\item The conditions of the lemma imply that $0 < 1 + \frac{q}{v} < 1$. Thus, for all $\varepsilon > 0$, there is $s > 1$ such that $-q-\varepsilon < v_{P_s} < -q$. Since $0 < q < 1$, there exists an $s > 1$ such that $P_s$ has type $B^+$. 
\item By the same argument as in~\ref{paths:(SM1)}, for every $\varepsilon > 0$, there exists $s > 1$ such that $-q-\varepsilon < v_{P_s} < -q$. Since $1 < q < 2$, there exists an $s > 1$ such that $P_s$ has type $B^-$. 
\item The conditions imply $q/v < -2$. Thus for any even $s$ we have $v_{P_s} > 0$. It follows that there is $s>1$ such that $P_s$ has type $A^+$.
\item Since $2 < q < -2v$, we have $-1 < 1+ \frac{q}{v} < 1$. Thus, for $s=2$, we have $v_{P_s} < -q$. In particular, $v_{P_s} < -2$. Thus, there exists an $s>1$ such that $P_s$ has type $A^-$.
\end{enumerate}
\vspace{-10pt}
\end{proof}

We will also require some less simple \ttgs. Many of these we take from \cite{Gol.Jer.ComplexityTutte} and \cite{Jac.Sok.ZeroFreeMultiTutte} where a similar technique is used. The following lemma is an intermediate step in the proof of Lemma 11 from~\cite{Gol.Jer.ComplexityTutte}.

\begin{lemma}\label{lem:GJ(7)}
If $q > 2$ and $-q < v < -2$, then there is a \ttg of type $B^+$ at $(q,v)$. 
\end{lemma}

\begin{proof}
Let $F$ be the \dip having two edges of weight $v$. Note that by~\eqref{eq:effWdipole} we have $v_F = v(v+2)>0$. Now let $G$ be the \ttg consisting of $s$ copies of $F$ and one edge of weight $v$ in series. By~\eqref{eq:effWpath} we have that 
\begin{equation}\label{eq:seriesv_Fandv}
1 + \frac{q}{v_G} = \(1 + \frac{q}{v_F} \)^s \(1 + \frac{q}{v}\).
\end{equation}
By the conditions of the lemma, we have that $1+ q/v <0$. Since $1+q/v_F > 1$, the right hand side of~\eqref{eq:seriesv_Fandv} tends to minus infinity as $s$ tends to infinity. It follows that there is $s$ such that $-1 < v_G < 0$, and for this $s$, $G$ is a \ttg of type $B^+$ at $(q,v)$.
\end{proof}

\begin{lemma}\label{lem:A-toA+}
Let $F$ be a \ttg with effective weight $v_F$ and let $q$ be a real variable. Also let $F^{(2)}$ be the graph consisting of two copies of $F$ in parallel. If $v_F(q) < -2$, then $F^{(2)}$ has effective weight $v_{F^{(2)}}(q) > 0$. If $-2 < v_F(q) < -1$, then $F^{(2)}$ has effective weight $-1 < v_{F^{(2)}}(q) < 0$.
\end{lemma}

\begin{proof}
Let $D$ denote the \dip with two edges of weight $v_F$. It may be verified that the effective weight of $F^{(2)}$ is equal to the effective weight of $D$. Thus, by~\eqref{eq:effWdipole}, we have $v_{F^{(2)}}(q) = v_F(q)(v_F(q)+2)$. If $v_F(q) < -2$, this is positive. If $-2 < v_F(q) < -1$, then $-1 < v_{F^{(2)}}(q) < 0$.
\end{proof}

In the following lemma we invoke Lemma 23 of~\cite{Gol.Jer.ComplexityTutte}, which uses the \ttg obtained from the Petersen graph by deleting an edge $xy$.

\begin{lemma}\label{lem:JGFlow}
If $2 < q < 4$ is non-integer and $v < -q$ then there is a \ttg of type $B^+$ at $(q,v)$. 
\end{lemma}

\begin{proof}
By the argument in Lemma 23 of~\cite{Gol.Jer.ComplexityTutte}, there is a \ttg $F$ satisfying $-q < v_F < 0$. If $-1 < v_F < 0$ then the result follows immediately. If $-q < v_F < -2$, then the result follows by Lemma~\ref{lem:GJ(7)}. If $-2 < v_F < -1$, then by Lemma~\ref{lem:A-toA+}, the \ttg $F^{(2)}$ formed by taking two copies of $F$ in parallel has effective weight $v_{F^{(2)}}$ satisfying $-1 < v_{F^{(2)}} < 0$ as required. Thus, it just remains to consider the cases when $v_F \in \{-1, -2\}$.

Let $J_s$ denote the graph consisting of $s$ copies of $F$ in series. The effective weight of $J_s$ is equal to the effective weight of $P_s$, where $P_s$ denotes the path with $s$ edges of weight $v_F$. Suppose $v_F = -1$. We have $v_{J_3} = \frac{-1}{q^2 - 3q + 3}$. It may be checked that for $q>2$ we have $-1 < v_{J_{3}} < 0$, so $J_3$ has type $B^+$ as required. Now suppose that $v_F = -2$. By~\eqref{eq:effWpath}, we have $1+ \frac{q}{v_{J_{s}}} = (1 - \frac{q}{2})^s$. Since $2 < q < 4$, we have that $-1< 1 - \frac{q}{2} < 0$, and so for any $\varepsilon > 0$, there exists a large and odd $s$ such that $-q< v_{J_{s}} < -q+ \varepsilon$. Thus we can ensure that $-q < v_{J_s} < -2$. The result now follows by an application of Lemma~\ref{lem:GJ(7)}.
\end{proof}

The following lemma uses a gadget based on large complete graphs and consequently, the resulting \ttg is non-planar.

\begin{lemma}\label{lem:JG(T)}\emph{\cite[Lemma 18]{Gol.Jer.ComplexityTutte}}
If $q > 2$ and $-1 < v < 0$ then there is a \ttg of type $A^-$ or $B^-$ at $(q,v)$. 
\end{lemma}

Recall that $\vplus(q)$ is the function describing the middle branch of the curve $v^3 - 2qv - q^2 = 0$ for $0 < q < 32/27$, and that $\vmin(q)$ is defined by $\vmin(q) = q/ \vplus(q)$ for $0 < q < 32/27$.

\begin{lemma}\label{lem:v-}
If $0 < q < 32/27$ and $-2 < v < \vmin$, then there is a \ttg $F$ of type $A^+$ at $(q,v)$. Furthermore, we can choose $F$ such that $F+xy$ is planar.
\end{lemma}

\begin{proof}
Let $H$ be the graph consisting of two edges of weight $v$ in parallel. We claim that $v_H > \vplus (q)$. By~\eqref{eq:effWdipole}, the effective weight $v_H$ of $H$ satisfies $v_H = v(v+2)$, which is a decreasing function of $v$ for $-2 < v < \vmin$. Note that $\frac{q}{v} ( \frac{q}{v} + 2 ) = v$ is precisely the equation satisfied by $\vplus(q)$. Thus $\vmin(q) ( \vmin(q)  + 2 ) = \vplus(q)$ for $0 < q < 32/27$. Since $v < \vmin $, it follows that $v_{H} = v(v+2) > \vmin(q) ( \vmin(q)  + 2 ) = \vplus(q)$ as claimed. Now by Lemmas~8.5(a) and 8.5(b) in~\cite{Jac.Sok.ZeroFreeMultiTutte}, there is a \ttg obtained from $H$ which has type $A^+$.
\end{proof}

\begin{lemma}\label{lem:GJ(8)}\emph{\cite[Lemma 12]{Gol.Jer.ComplexityTutte}}
If $q > 32/27$ and $-2 < v < -q/2$, then there is a \ttg $F$ of type $A^+$ at $(q,v)$. Furthermore, we can choose $F$ such that $F+xy$ is planar.
\end{lemma}

We now combine the results of this section and Section~\ref{sec:Strat} to prove Theorem~\ref{thm:main}.

\begin{proof}[Proof of Theorem~\ref{thm:main}]
We first show that the zeros of the Tutte polynomial are dense in regions I - IX. By Corollary~\ref{cor:strategy}, it suffices to show that for each point $(q,v)$ in Regions I - IX, there exists a complementary pair of graphs. In regions I - V, a single edge of weight $v$ has type $A^-$. By Lemma~\ref{lem:A-toA+}, the graph consisting of two such edges in parallel has type $A^+$. Thus, in regions I - V, it only remains to find a \ttg of type $B^+$ or $B^-$.

\begin{list}{}{}
\item[Region I:]  By Lemma~\ref{lem:paths}\ref{paths:(5)}, there is a \ttg of type $B^+$.
\item[Region II:] By Lemma~\ref{lem:paths}\ref{paths:(SM1)}, there is a \ttg of type $B^+$.
\item[Region III:] By Lemma~\ref{lem:paths}\ref{paths:(SM2)}, there is a \ttg of type $B^-$.
\item[Region IV:] By Lemma~\ref{lem:JGFlow}, there is a \ttg of type $B^+$.
\item[Region V:] By Lemma~\ref{lem:GJ(7)}, there is a \ttg of type $B^+$.

\end{list}

We deal with the remaining regions individually.

\begin{list}{}{}
\item[Region VI:] For $(q,v)$ in region VI, a single edge of weight $v$ has type $B^-$. By Lemma~\ref{lem:paths}\ref{paths:(VI)}, there exists a \ttg of type $A^-$. By Lemma~\ref{lem:A-toA+}, taking two copies of this graph in parallel gives a \ttg of type $A^+$ as required.
\item[Region VII:] A single edge of weight $v$ has type $B^+$. By Lemma~\ref{lem:JG(T)} there is a \ttg $F$ of type $A^-$ or $B^-$ at $(q,v)$. If $F$ has type $A^-$ then we are done. If $F$ has type $B^-$, then the effective weight $v_F$ of $F$ satisfies $-2 < v_F < -1$. Thus, the point $(q,v_F)$ lies in region VI or V$^*$. If $(q,v_F) \in VI$, then we use the argument for region VI to obtain a \ttg of type $A^+$. If $(q,v_F) \in \text{V}^*$, then we use Lemma~\ref{lem:paths}\ref{paths:(1)} to obtain a \ttg of type A$^+$.
\item[Region VIII:] A single edge of weight $v$ has type $B^-$. By Lemma~\ref{lem:v-}, there is a \ttg of type $A^+$ at $(q,v)$.
\item[Region IX:] A single edge of weight $v$ has type $B^-$. By Lemma~\ref{lem:GJ(8)}, there is a \ttg of type $A^+$ at $(q,v)$.
\end{list}

We note that in regions I, II, III, V, VIII and IX, each complementary pair that we use is planar. Thus by Corollary~\ref{cor:duality}, the zeros of the Tutte polynomials of planar graphs are also dense in the regions I$^*$, II$^*$, III$^*$, V$^*$, VIII$^*$ and IX$^*$.
\end{proof}

We briefly remark on the region in which we have been unable to prove density, namely the points satisfying $q > 4$ and $v < -q$. For $(q,v)$ in this region, the sequence of paths $P_s$, $s \in \mathbb{N}$ have effective weights converging to the point $-q$. Along the line $v = -q$, the multivariate Tutte polynomial is nothing other than the flow polynomial $F_G(q)$ multiplied by a prefactor. Goldberg and Jerrum have shown that if $G$ is a graph and $xy \in E(G)$ such that $F_G(q)$ and $F_{G-xy}(q)$ have opposite signs, then it is possible to implement a weight $v'$ satisfying $-q < v' < 0$. Using an argument similar to that of Lemma~\ref{lem:JGFlow}, it would then be possible to find a \ttg which has type $B^+$ at $(q,v)$. It is conjectured~\cite{Jac.Sal.False} that there exists $q_0 \in \mathbb{R}$ such that $F_G(q) > 0$ for all $2$-connected graphs $G$ and all $q > q_0$. Thus it seems unlikely that this technique can be used to prove density for all $q > 4$.

The dual of the unsolved region lies inside region VII. Unfortunately, the graphs we used to prove density in region VII are non-planar, and so we cannot use duality as we have done above. However, if we allow ourselves to use all matroids instead of all graphs then we can apply the duality argument, since every matroid has a dual. It is easy to define the Tutte polynomial for matroids by replacing the term $q^{k(A)}$ with $q^{r(G) - r(A) + 1}$ where $r$ is the rank function, see~\cite{Sok.MultiTutte}. 

\section{Acknowledgements}
The authors would like to thank Schloss Dagstuhl – Leibniz Center for Informatics and the organisers of the ``Graph Polynomials" seminar held June 12 – 17, 2016, where this work was initiated.

\bibliographystyle{plain}                     
\bibliography{References}   
\end{document}